\documentclass[a4paper,reqno,oneside,11pt]{amsart}
\usepackage{fullpage}
\usepackage{setspace}
\usepackage{datetime}

\usepackage{amsmath}
\usepackage{amsfonts}
\usepackage{amssymb}
\usepackage{verbatim,enumitem,graphics}
\usepackage{microtype}
\usepackage{xcolor}
\usepackage[backref=page]{hyperref}
\hypersetup{%
    colorlinks,
    linkcolor={red!50!black},
    citecolor={green!50!black},
    urlcolor={blue!50!black}
}
\usepackage{dsfont}

\usepackage[utf8]{inputenc}

\usepackage[abbrev,msc-links,backrefs]{amsrefs} 
\usepackage{doi}

\renewcommand{\PrintDOI}[1]{\doi{#1}}

\usepackage{fnpct}

\usepackage{enumitem}
\usepackage{MnSymbol}

\setenumerate{leftmargin=*} 

\usepackage{enumitem}

\def\nlabel{\upshape({\itshape \arabic*\,})} 

\newtheorem{theorem}{Theorem}[section]
\newtheorem{lemma}[theorem]{Lemma}
\newtheorem{proposition}[theorem]{Proposition}
\newtheorem{fact}[theorem]{Fact}

\newtheorem{corollary}[theorem]{Corollary}
\newtheorem{conjecture}[theorem]{Conjecture}

\newcommand{\oldqed}{}
\def\endofClaim{\hfill\scalebox{.6}{$\Box$}}

\newcommand{\cF}{\mathcal{F}}

\newcommand{\cP}{\mathcal{P}}

\newcommand{\cK}{\mathcal{K}}
\newcommand{\cI}{\mathcal{I}}
\newcommand{\cT}{\mathcal{T}}

\newcommand{\PP}{\mathbb{P}}
\newcommand{\RR}{\mathbb{R}}
\newcommand{\Bin}{\mathrm{Bin}}
\newcommand{\EE}{\mathbb{E}}

\newcommand{\eps}{\varepsilon}

\newcommand{\cC}{\mathcal{C}}

\newcommand{\upth}{^{\rm th}}
\newcommand{\upst}{^{\rm st}}
\let\subset\subseteq

\DeclareMathOperator{\polylog}{{\rm polylog}}

\usepackage{lineno}
\newcommand*\patchAmsMathEnvironmentForLineno[1]{%
\expandafter\let\csname old#1\expandafter\endcsname\csname #1\endcsname
\expandafter\let\csname oldend#1\expandafter\endcsname\csname end#1\endcsname
\renewenvironment{#1}%
{\linenomath\csname old#1\endcsname}%
{\csname oldend#1\endcsname\endlinenomath}}%
\newcommand*\patchBothAmsMathEnvironmentsForLineno[1]{%
\patchAmsMathEnvironmentForLineno{#1}%
\patchAmsMathEnvironmentForLineno{#1*}}%
\AtBeginDocument{%
\patchBothAmsMathEnvironmentsForLineno{equation}%
\patchBothAmsMathEnvironmentsForLineno{align}%
\patchBothAmsMathEnvironmentsForLineno{flalign}%
\patchBothAmsMathEnvironmentsForLineno{alignat}%
\patchBothAmsMathEnvironmentsForLineno{gather}%
\patchBothAmsMathEnvironmentsForLineno{multline}%
}

\title{Near-perfect clique-factors in sparse pseudorandom
  graphs}

\author[J. Han]{Jie Han}
\author[Y. Kohayakawa]{Yoshiharu Kohayakawa}
\author[Y. Person]{Yury Person}

\thanks{%
  JH~is supported by FAPESP
  (2014/18641-5, 2013/03447-6).  YK is partially supported by FAPESP
  (2013/03447-6) and CNPq (310974/2013-5, 459335/2014-6).  YP was
  supported by DFG grant PE 2299/1-1.  The cooperation of the authors
  was supported by a joint CAPES-DAAD PROBRAL project (Proj.\
  no.~430/15, 57350402, 57391197).  }

\address{Instituto de Matem\'atica e Estat\'{\i}stica, Universidade de
    S\~ao Paulo, S\~ao Paulo, Brazil}
\email{\{jhan|yoshi\}@ime.usp.br}

\address{Institut f\"ur Mathematik, Technische Universit\"at Ilmenau, 98684 Ilmenau, Germany}
\email{yury.person@tu-ilmenau.de}

\begin{document}
 \onehalfspace
\shortdate
\yyyymmdddate
\settimeformat{ampmtime}
\date{\today, \currenttime}
\footskip=28pt
\allowdisplaybreaks

\begin{abstract}
  We prove that, for any $t\ge 3$, there exists a constant $c=c(t)>0$
  such that any $d$-regular $n$-vertex graph with the second largest
  eigenvalue in absolute value~$\lambda$ satisfying
  $\lambda\le c d^{t-1}/n^{t-2}$ contains vertex-disjoint copies of $K_t$ covering all but at most
  $n^{1-1/(8t^4)}$ vertices.  This provides further support for
  the conjecture of Krivelevich, Sudakov and Sz\'abo [\emph{Triangle
    factors in sparse pseudo-random graphs}, Combinatorica \textbf{24}
  (2004), pp.~403--426] that $(n,d,\lambda)$-graphs with
  $n\in 3\mathbb{N}$ and $\lambda\leq cd^{2}/n$ for a suitably small
  absolute constant~$c>0$ contain triangle-factors.
\end{abstract}

\maketitle

\section{Introduction}
\label{sec:intro}
The study of conditions under which certain spanning or almost
spanning structures are forced in random or pseudorandom graphs is one
of the central topics in extremal graph theory and in random graphs.
 
An $(n,d,\lambda)$-graph is an $n$-vertex $d$-regular graph whose
second largest eigenvalue in absolute value is at most~$\lambda$.
Graphs with $\lambda\ll d$ are considered to be \textit{pseudorandom},
i.e., they behave in certain respects as random graphs do; for
example, the edge count between `not too small' vertex subsets is
close to what one sees in random graphs of the same density.  As
usual, let $e(A,B)=e_G(A,B)$ denote the number of pairs
$(a,b)\in A\times B$ so that $ab$ is an edge of~$G$ (note that edges
in $A\cap B$ are counted twice).  The following result makes what we
discussed above precise.

\begin{theorem}[Expander mixing lemma~\cite{AC88}]
  \label{thm:EML}
  If $G$ is an $(n,d,\lambda)$-graph and $A$, $B\subseteq V(G)$, then
  \begin{equation}
    \left|e(A,B)-\frac{d}{n}|A||B|\right|<\lambda\sqrt{|A||B|}.\label{eq:EML}
  \end{equation}
\end{theorem}

As starting points to the extensive literature on pseudorandom graphs,
the reader is refereed to, e.g., \cite{KS06}, \cite{CFZ12}
or~\cite[Chapter~9]{AS16}.

It is an interesting problem to understand optimal or asymptotically
optimal conditions on the parameter~$\lambda$ in terms of~$d$ and~$n$
that force an $(n,d,\lambda)$-graph to possess a desired property.  To
demonstrate the optimality of a condition, one needs to show the
existence of an $(n,d,\lambda)$-graph that certifies that the
condition is indeed optimal.

Unfortunately, there are very few examples certifying optimality.  A
celebrated example is due to Alon, who showed~\cite{Alon94} that there
are $(n,d,\lambda)$-graphs that are $K_3$-free and yet satisfy
$\lambda=cd^2/n$ for some absolute constant~$c>0$.  This is in
contrast with the fact that, as it follows easily from the expander
mixing lemma above, for, say, $\lambda\le 0.1 d^2/n$, any
$(n,d,\lambda)$-graph contains a triangle (in fact, every vertex lies
in a triangle).  It turns out that $(n,d,\lambda)$-graphs with
$\lambda=\Theta(d^2/n)$ must satisfy $d=\Omega(n^{2/3})$.  The
construction of Alon~\cite{Alon94} provides an example of the
essentially sparsest possible $K_3$-free $(n,d,\lambda)$-graph with
$d=\Theta(n^{2/3})$ and $\lambda=\Theta(n^{1/3})$. The other known example 
is a generalization of this construction by Alon and Kahale~\cite{AK98} (see also~\cite[Section~3]{KS06}) to graphs without 
odd cycles of length at most $2\ell+1$.

Our focus here is on spanning or almost spanning structures in
$(n,d,\lambda)$-graphs.  One of the simplest spanning structures is
that of a perfect matching. Alon, Krivelevich and Sudakov
(see~\cite{KS06}) proved that $(n,d,\lambda)$-graphs with
$\lambda\le d-2$ and~$n$ even contain perfect matchings.
\textit{Factors} generalize perfect matchings: for a graph~$F$, an
\emph{$F$-factor} in a graph~$G$ is a collection of vertex-disjoint
copies of~$F$ in~$G$ whose vertex sets cover~$V(G)$ (this requires
that $v(G):=|V(G)|$ should be divisible by $v(F)$).  Motivated by the
study of spanning structures in graphs, Krivelevich, Sudakov and
Szab\'o~\cite{KSS04} proved that $(n,d,\lambda)$-graphs with
$\lambda=o\left(d^3/(n^2\log n)\right)$ contain a triangle-factor
if~$3\mid n$.
 
A \textit{fractional} triangle-factor in a graph $G=(V,E)$ is a
non-negative weight function $f$ on the set $\cK_3(G)$ of all
triangles $T$ of $G$, such that, for every~$v\in V$, we have
$\sum_{T\colon v\in V(T)}f(T)=1$.  Krivelevich, Sudakov and Szab\'o
further proved~\cite{KSS04} that $(n,d,\lambda)$-graphs with
$\lambda\le 0.1 d^2/n$ admit a fractional triangle-factor.  Moreover,
they conjectured the following.
 
\begin{conjecture}[Conjecture~7.1 in~\cite{KSS04}]
  \label{conj:KSS}
  There exists an absolute constant $c>0$ such that if
  $\lambda\le cd^2/n$, then every $(n,d,\lambda)$-graph~$G$ on
  $n\in 3\mathbb{N}$ vertices has a triangle-factor.
\end{conjecture}

The \emph{$t\upth$ power $H^t$} of a graph $H$ is the graph on the
vertex set $V(H)$ where~$uv$ ($u\neq v$) is an edge if there is a
$u$-$v$-path of length at most $t$ in $H$.  Since the $(t-1)\upst$
power of a Hamilton cycle contains a $K_t$-factor if $t\mid n$, powers
of Hamilton cycles are also of interest when investigating
clique-factors.

Allen, B\"ottcher, H\`an and two of the authors~\cite{ABHKP17} proved
that, if $\lambda=o(d^{3t/2}n^{1-3t/2})$ and~$t\geq3$, then any
$(n,d,\lambda)$-graph contains the $t\upth$ power of the Hamilton
cycle (and thus a $K_{t+1}$-factor if $(t+1)\mid n$).  In the case
$t=2$, it was further proved in~\cite{ABHKP17} that the condition
$\lambda=o\left(d^{5/2}/n^{3/2}\right)$ suffices to guarantee squares
of Hamilton cycles, and thus $K_3$-factors, improving over the
aforementioned result of Krivelevich, Sudakov and Szab\'o.  For very
recent progress, due to Nenadov~\cite{Nenadov18}, see
Section~\ref{sec:remarks} below.
 
The construction of Alon of $K_3$-free $(n,d,\lambda)$-graphs shows
that the condition on~$\lambda$ in Conjecture~\ref{conj:KSS} cannot be
weakened.  The result from~\cite{KSS04} on the existence of
fractional triangle-factors supports Conjecture~\ref{conj:KSS}.  As a
further evidence in support of that conjecture we prove here the
following result.

\begin{theorem}[Main result\footnote{This result appears in an
    extended abstract in the Proceedings of Discrete Mathematics Days
    2018 (Sevilla)~\cite{HKP18a}.}\footnote{Theorem~\ref{thm:main}
    answers a question raised by Nenadov (see of~\cite[Concluding
    remarks]{Nenadov18}).}]
  \label{thm:main}
  For any $t\ge 3$ there is $n_0>0$ for which the following holds.
  Every $(n,d,\lambda)$-graph~$G$ with $n\ge n_0$ and
  $\lambda\le (1/(50t4^{t-2})) d^{t-1}/n^{t-2}$ contains
  vertex-disjoint copies of $K_t$ covering all but at most
  $n^{1-1/(8t^4)}$ vertices of~$G$.
\end{theorem}

We remark that, under the condition $\lambda\le cd^{t-1}/n^{t-2}$ for
some appropriate $c=c(t)>0$, Krivelevich, Sudakov and
Szab\'o~\cite{KSS04} proved that any $(n,d,\lambda)$-graph contains a
fractional $K_t$-factor.

A na\"ive approach to proving Theorem~\ref{thm:main} is to pick
cliques~$K_t$ one after another, each vertex-disjoint from the
previous ones, by appealing to the pseudorandomness of~$G$ via the
expander mixing lemma.  However, even for triangles, if
$\lambda=cd^2/n$, then all what one gets this way is that~$G$ has
$(1-c)n/3$ vertex-disjoint triangles: one can see that a set of~$cn$
vertices in~$G$ induces a graph of average degree roughly~$cd$, but the
condition on~$\lambda$ and the expander mixing lemma do not guarantee
that sets of size roughly~$cd$ contain an edge, and hence we do not know
whether~$cn$ vertices necessarily span a triangle.  Thus our na\"ive
greedy approach will get stuck leaving~$cn$ vertices uncovered.  What
our result establishes is that, even for some absolute constant~$c>0$,
we can cover all but~$o(n)$ vertices of~$G$ by vertex-disjoint copies
of~$K_3$.  Moreover, $o(n)$ can be taken to be of the
form~$n^{1-\eps}$ for some $\eps>0$.  We have restricted ourselves to
triangles in this paragraph, but a similar reasoning applies to
general cliques~$K_t$ as well.

Now let~$p=d/n$ and suppose~$G=(V,E)$ is an $(n,d,\lambda)$-graph
with $\lambda\le cd^2/n$.  Inequality~(\ref{eq:EML}) implies that
\begin{equation}
  \label{eq:epsreg} 
  \left|\frac{e(A,B)}{|A||B|}-p\right|<{cp^2n\over\sqrt{|A||B|}}
  \le c^{1/2}
\end{equation}
for all~$A$, $B\subseteq V$ with $|A|$, $|B|\ge c^{1/2}n$.  Let us now
focus on the case in which~$d$ is linear in~$n$, that is, $p=d/n$~is a
constant independent of~$n$.  The powerful blow-up lemma of Koml\'os,
S\'ark\"ozy and Szemer\'edi~\cite{KSS_bl} implies that, if~$c$ is
small enough in comparison with~$p$ and~$1/t$, then \emph{any}
graph~$G=(V,E)$ on~$n$ vertices with minimum degree at least~$pn$ that
satisfies~(\ref{eq:epsreg}) contains a $K_t$-factor as long as
$t\mid n$.  Thus, Conjecture~\ref{conj:KSS} holds for dense graphs.

We remark that the blow-up lemmas for sparse graphs developed recently
by Allen, B\"ottcher, H\`an and two of the authors~\cite{ABHKP14b}
provide bounds on $\lambda$ to establish the existence of
$K_t$-factors, but those bounds are worse than those
from~\cite{ABHKP17} discussed above.

Throughout the paper floor and ceiling signs are omitted for the sake
of readability.  For graph theory terminology and notation we refer
the reader to Bollob\'as~\cite{B98}.

This paper is organized as follows. In Section~\ref{sec:tools} we
collect some of the necessary tools and prove auxiliary results. In
Section~\ref{sec:outline} we provide an overview of the proof of
Theorem~\ref{thm:main}, which splits into two cases (the `dense' and
`sparse' cases).  We deal with these two cases in
Sections~\ref{sec:dense} and~\ref{sec:sparse}, separately, and then
prove Theorem~\ref{thm:main} in Section~\ref{sec:main}.  Finally we
give some concluding remarks in Section~\ref{sec:remarks}.


 
\section{Tools and auxiliary results}\label{sec:tools}

\subsection{Probabilistic techniques} 

We shall use the following theorem of Kostochka and R\"odl~\cite{KoRo}
(see also R\"odl~\cite{VR85} and Alon and
Spencer~\cite[Theorem~4.7.1]{AS16}), which asserts the existence of an
almost perfect matching in `pseudorandom' hypergraphs.

\begin{theorem}\label{thm:Pippenger}
  Let integers $t\ge 3$ and $k\ge 8$ and real numbers~$\delta'$
  and~$\gamma$ with $0<\delta',\,\gamma <1$ be fixed.  Then there
  exists $D_0$ such that the following holds for $D\ge D_0$.  Let $H$
  be a $t$-uniform hypergraph on a set $V$ of $n$ vertices such that
  \begin{enumerate}[label=\nlabel]
  \item for all vertices $v\in V$, we have $D - k\sqrt{D\log D}\le
    \deg_H(v)\le D$ and
  \item for any two distinct vertices~$u$ and~$v\in V$, we have
    $\deg_H(u,v)\le C<D^{1-\gamma}$.
  \end{enumerate}
  Then $H$ contains a matching covering all but $O(n
  (C/D)^{(1-\delta')/(t-1)})$ vertices. 
\end{theorem}

We shall use the following concentration results.

\begin{theorem}[Chernoff bounds {\cite[Corollary~2.4 and Theorems~2.8
    and~2.10]{JLR00}}]
  \label{thm:chernoff}
  Suppose $X$ is a sum of a collection of independent Bernoulli random
  variables.  Then, for $\delta\in(0,3/2)$, we have
  \[\PP\big(X>(1+\delta)\EE  X\big)<e^{-\delta^2\EE X/3}\quad\text{and}\quad
    \PP\big(X<(1-\delta)\EE  X\big)<e^{-\delta^2\EE X/2}
    \,.\]
  Moreover, for any $t\ge 6\EE X$, we have
  \[\PP\big(X\ge \EE X+t)\le e^{-t}\,.\]
\end{theorem}

For a graph $G=(V,E)$ we denote by $G_p$ the spanning random subgraph
of~$G$ in which each edge from $E$ is included with probability~$p$,
independently of all other edges.

\begin{theorem}[{Janson's inequality~\cite[Theorem~2.14]{JLR00}}]
  \label{thm:Janson}
  Let $p \in (0,1)$ be given and consider a family
  $\{ H_i \}_{i \in \cI}$ of subgraphs of a graph~$G$.
  For each $i \in \cI$, let~$X_i$ denote the indicator random variable
  for the event that $H_i \subseteq G_p$.  Write $H_i \sim H_j$ for
  each ordered pair $(i,j) \in \cI \times \cI$ such that
  $E(H_i) \cap E(H_j) \neq \emptyset$.
  Let~$X = \sum_{i \in \cI} X_i$.  Then
  $\mathbb{E}[X] = \sum_{i \in \cI} p^{e(H_i)}$.  Furthermore, let
  \begin{equation}
    \label{eq:Delta_def}
    \Delta = \sum_{H_i \sim H_j} \mathbb{E}[X_i X_j]
    = \sum_{H_i \sim H_j} p^{e(H_i) + e(H_j) - e(H_i \cap H_j)}.
  \end{equation}
  Then, for any $0 < \gamma < 1$, we have
  \begin{equation}
    \label{eq:Janson}
    \mathbb{P} [X \le (1-\gamma) \mathbb{E}[X]] \le \exp
    \left(-\frac{\gamma^2 \mathbb{E}[X]^2}{2\Delta} \right). 
  \end{equation}
\end{theorem}

\subsection{Linear programming techniques}
\label{sec:LP}
We shall consider \emph{weighted graphs}~$(G,w)$ where~$G=(V,E)$ is a
graph and $w\colon E\to [0,1]$ is a function on its edges.
If~$w\equiv 1$, we identify $(G,w)$ with~$G$.  For every vertex
$v\in V$, we define its \emph{weighted degree} $\deg_w(v)$ to be
$\sum_{u\in N(v)} w(uv)$.

Let $\cK_t(G)$ denote the set of all copies of $K_t$ in $G$. A functon
$f\colon \cK_t(G) \to[0,1]$ is called a \emph{fractional $K_t$-factor}
if $\sum_{T\in\cK_t(G)\colon V(T)\ni v}f(T)= 1$ for every $v\in V$,
and for every edge $uv\in E(G)$ one has
\begin{equation}\label{eq:ftf_weight}
  \sum_{T\in \cK_t(G)\colon V(T)\supset\{u,v\}} f(T)\le w(uv).
\end{equation}
A fractional $K_t$-factor in a weighted graph~$(G,w)$ thus `respects'
the weight function~$w$.  We remark that the definition of a
fractional $K_t$-factor in a graph generalizes the definition of a
fractional triangle-factor from the introduction in a straightforward
way, and it is itself generalized by the above one for weighted
graphs, since condition~\eqref{eq:ftf_weight} is satisfied in the case
of fractional $K_t$-factors in unweighted graphs because of our
convention that $w\equiv 1$ in that case.

We shall use the duality theorem of linear programming; see
e.g.~\cite{MG07}.  The maximum weight of a fractional
\textit{$K_t$-matching} in~$(G,w)$ is given by the following linear
programme:
\begin{align} 
\begin{split}
\label{opt:primal}
\max\,\,& \sum_{T\in\cK_t(G)}f(T)\\
& \sum_{T\in\cK_t(G)\colon V(T)\ni v} f(T)\le 1\quad (\forall\, v)\\
& \sum_{T\in\cK_t(G)\colon V(T)\supset \{u, v\}} f(T)\le w(uv)\quad (\forall\, uv\in E(G))\\
& f(T)\ge 0 \quad (\forall\, T\in\cK_t(G))
\end{split}
\end{align}

The dual of~\eqref{opt:primal} is the following linear programme:
\begin{align} 
\begin{split}
\label{opt:dual}
\min\,\,& \sum_{v\in V}g(v)+\sum_{uv\in E(G)} h(uv)w(uv)\\
& \sum_{v\colon v\in V(T)} g(v)+\sum_{uv\in E(T)} h(uv)\ge 1\quad (\forall\, T\in\cK_t(G))\\
& g(v)\ge 0 \quad (\forall\, v)\\
& h(uv)\ge 0 \quad (\forall\, uv\in E(G))\\
\end{split}
\end{align}
Both linear programmes above are clearly feasible, and therefore the
duality theorem tells us that both admit optimal solutions and that,
moreover, these optimal values are equal.  Given optimal solutions~$f$
and~$(g,h)$, the complementary slackness conditions tell us that if
$g(v)>0$ then the corresponding inequality
$\sum_{T\in\cK_t(G)\colon V(T)\ni v} f(T)\le 1$ in the primal
linear programme~\eqref{opt:primal} holds with equality, i.e.,
$\sum_{T\in\cK_t(G)\colon V(T)\ni v} f(T)= 1$.

Let~$t^*(G,w)$ be the optimum value of~\eqref{opt:primal}
and~\eqref{opt:dual}.  For a collection~$\cF$ of vertex-disjoint
copies of $K_t$ from $\cK_t(G)$, let
$w(\cF):=\sum_{T\in \cF} \min_{uv\in E(T)}w(uv)$.  Furthermore,
let~$t(G,w)$ be the maximum possible value of~$w(\cF)$ for such a
collection~$\cF$.  Write~$|g|$ for $\sum_{v\in V}g(v)$ and $|h|$ for
$\sum_{uv\in E(G)} h(uv)w(uv)$ (see the objective function
in~\eqref{opt:dual}).
 
The following proposition collects some useful properties of linear
programmes~\eqref{opt:primal} and~\eqref{opt:dual}.  A variant for
unweighted uniform hypergraphs and fractional matchings was first
stated and proved by Krivelevich in~\cite[Proposition~2]{Kriv96}.

\begin{proposition}
  \label{prop:LPw}
  Let $t\ge 3$ be given and let $(G,w)$ be a weighted graph.
  Suppose~$G=(V,E)$.  Then the following hold.
  \begin{enumerate}[label=\nlabel]
  \item $t^*(G,w)\ge t(G,w)$.  \label{it:tts}
  \item $t^*(G,w)\le |V|/t$.  Furthermore, if~$t^*(G,w)=|V|/t$,
    then~$(G,w)$ has a fractional $K_t$-factor.  \label{it:ts_bound}
  \item If $g\colon V\to \RR_{\ge 0}$ and $h \colon E\to \RR_{\ge 0}$
    form a feasible solution to~\eqref{opt:dual}, then for every
    subset $U\subseteq V$ the functions $g':=g\restriction_U$ and
    $h':=h\restriction_{E\cap \binom{U}{2}}$ form a feasible solution
    to~\eqref{opt:dual} with $G[U]$ in place of $G$; in particular we
    have $|g'|+|h'|\ge t^*(G[U],w)$.  \label{it:restr_w}
  \item If $g\colon V\to\RR_{\ge 0}$ and $h \colon E\to \RR_{\ge 0}$
    form an optimal solution to~\eqref{opt:dual}, then
    $t^*(G,w)\ge |V_1|/t$, where
    $V_1:=\{v\in V\colon g(v)>0\}$.  \label{it:lower_bound}
\end{enumerate}
\end{proposition}
\begin{proof}
  Let~$\cF$ be a collection of vertex-disjoint copies of~$K_t$.  Let
  $f(T):=\min_{uv\in E(T)}w(uv)$ for every $T\in \cF$ and $f(T):=0$
  for every $T\in \cK_t(G)\setminus \cF$.  Clearly~$f$
  satisfies~\eqref{opt:primal}, whence~\ref{it:tts} follows.

  From $\sum_v\sum_{T\in\cK_t(G)\colon V(T)\ni v} f(T)\le |V|$ and
  $\sum_v\sum_{T\in\cK_t(G)\colon V(T)\ni v} f(T)=t\sum_{T\in\cK_t(G)}
  f(T)$, assertion~\ref{it:ts_bound} follows immediately.

  It is also clear that restricting a feasible solution
  of~\eqref{opt:dual} to a subset $U\subseteq V$, we obtain functions
  for which the relevant inequalities hold.  Since the value of a
  feasible solution of the dual~\eqref{opt:dual} is always at least
  the value of a feasible solution of the primal
  programme~\eqref{opt:primal} (if both are feasible, which is the
  case here), we obtain $|g'|+|h'|\ge t^*(G[U],w)$.  Thus,
  assertion~\ref{it:restr_w} follows.

  From the complementary slackness conditions we know that, whenever
  $g(v)>0$, we have $\sum_{T\in\cK_t(G)\colon V(T)\ni v} f(T)=
  1$. This implies that
  \[
    t^*(G,w)=\sum_{T\in\cK_t(G)} f(T)=\frac{1}{t}\sum_{v\in
      V}\sum_{T\in\cK_t(G)\colon V(T)\ni v} f(T)\ge 
    \frac1t\sum_{v\in V_1}1={|V_1|\over t},
  \]
  which proves~\ref{it:lower_bound}. 
\end{proof}




\subsection{Applications of linear programming to weighted graphs}
\label{sec:P_ndl}
In this section we provide an auxiliary lemma that will help us verify
later that certain weighted subgraphs of $(n,d,\lambda)$-graphs
possess fractional $K_t$-factors.
The main result of this subsection, Corollary~\ref{cor:ff},
generalizes results from Krivelevich, Sudakov and
Szab\'o~\cite[Section~5]{KSS04} to weighted graphs and fractional
$K_t$-factors.  Our proofs follow their proof strategy.

We first state a simple fact.  Given $\alpha\ge 0$ and a weighted graph
$(G,w)$, we call an edge $uv\in E(G)$ \emph{$\alpha$-rich} if
$w(uv)\ge 1-\alpha$.  Similarly, we call a copy $T$ of $K_t$ in $G$
\emph{$\alpha$-rich} if $w(T)\ge 1-\alpha$, where
$w(T):=\min_{uv\in E(T)}w(uv)$.

\begin{fact}\label{fact:a_rich}
  Let $(G,w)$ be a weighted $d$-regular graph.  If
  $\sum_{u\in N(v)} w(uv)\ge d(1-\alpha^2)$ for a vertex $v$, then $v$
  is incident to at least $(1-\alpha)d$ many $\alpha$-rich edges.
\end{fact}
\begin{proof}
  Let $m$ be the number of non-$\alpha$-rich edges incident to a
  vertex~$v$.  Then we have $\sum_{u\in N(v)} w(uv)\le d-m\alpha$.  If
  $\sum_{u\in N(v)} w(uv)\ge d(1-\alpha^2)$, then
  $d\alpha^2\ge m\alpha$ and thus $m\le \alpha d$.
\end{proof}

Let a weighted graph~$(G,w)$ with~$G=(V,E)$ be given.  In what
follows, we shall consider the spanning subgraph~$H=(V,F)$ of~$G$,
with~$F\subset E$ the set of $\alpha$-rich edges of~$(G,w)$,
where~$\alpha$ will be chosen suitably.

Let $t\ge 3$ and $0<D,\,D'<n$ be given.  An $n$-vertex graph $G$ has
property $\cP(t, D, D', n)$ if the following holds.  For every subset
$U\subset V(G)$ of cardinality $|U|\ge n-D$ and for every subset
$U_0\subset U$ with $|U_0|=D/t$, there exists a family $\cT_0$ of
at least $D'/(t-1)$ copies of~$K_t$ in $G[U]$ with the following
properties:
\begin{enumerate}
\item $V(T)\subset U$ for every $T\in \cT_0$,\label{pr:one}
\item $|V(T)\cap U_0|=1$ for every $T\in \cT_0$,
\item $V(T)\cap V(T')\subset U_0$ for any distinct $T$
  and~$T'\in\cT_0$.  \label{pr:three}
\end{enumerate}

The copies of $K_t$ specified above thus have the property that they
are all edge-disjoint and intersect pairwise in at most one vertex,
which must be from~$U_0$.

\begin{lemma}\label{lem:large_case_tw}
  Let~$(G,w)$ be a weighted graph and let $\alpha\in[0, 1/(20t))$ be
  given.  Suppose~$H$ is the spanning subgraph of~$G$ formed by the
  $\alpha$-rich edges in~$(G,w)$.  Suppose~$H$ is such that
  \begin{enumerate}[label=\nlabel]
  \item every $0.11n/t$ vertices span a copy of $K_t$ in $H$ and
  \item $H$ has property $\cP(t, D, 0.2n, n)$ for some $D\le n/2$. 
  \end{enumerate}
  Then $t^*(G[U],w)\ge t^*(H[U],w)\ge(|U|-D/t)/t$ for every
  $U\subseteq V=V(G)$ with $|U|\ge n-D$. 
\end{lemma}

\begin{proof}
  Let~$U$ as in the statement of the lemma be given.  Let $g$ and $h$
  form an optimal solution of the linear programme~\eqref{opt:dual}
  applied to $H[U]$ instead of $G$.  Let $U'=\{u\in U\colon
  g(u)=0\}$. If $|U'|\le D/t$ then we have by
  Proposition~\ref{prop:LPw}~\ref{it:lower_bound} that
  $t^*(H[U],w)\ge|U\setminus U'|/t\ge(|U|-D/t)/t$ and we are done.

  We now assume that $|U'|>D/t$ and derive a contradiction.  First we
  fix a subset $U_0\subset U'$ with $|U_0|=D/t$ and then we consider a
  family $\cT_0$ of cardinality $0.2n/(t-1)$ as given by
  property~$\cP$.  Let
  $W:=\left(\bigcup_{T\in \cT_0}V(T)\right)\setminus U_0$.  We have
  $|W|=0.2n$ and, moreover,
  \begin{multline*}
    \sum_{v\in W}g(v)+\sum_{T\in \cT_0}\sum_{uv\in E(T)}h(uv)w(uv)\ge
    \sum_{v\in W}g(v)+\sum_{T\in \cT_0}\sum_{uv\in
      E(T)}h(uv)(1-\alpha)\\ 
    \ge (1-\alpha)\left(\sum_{v\in W}g(v)+\sum_{T\in \cT_0}\sum_{uv\in
        E(T)}h(uv)\right) 
    = (1-\alpha)\sum_{T\in \cT_0}\left(\sum_{v\in W\cap
        V(T)}g(v)+\sum_{uv\in E(T)}h(uv)\right)\\ 
    \ge (1-\alpha)|W|/(t-1)
  \end{multline*}
  because $g$ and $h$ form an optimal solution of~\eqref{opt:dual},
  $g\restriction_{U_0}\equiv0$ and
  $\sum_{v\colon v\in V(T)} g(v)+\sum_{V(T)\supset\{u,v\}} h(uv)\ge 1$
  for every $T\in\cK_t(H[U])$.
 
  Since every $0.11n/t$ vertices of $H$ span a copy of $K_t$, we find
  at least $(|U\setminus W|-0.11n/t)/t$ vertex-disjoint copies of
  $K_t$ in $H[U\setminus W]$.  Proposition~\ref{prop:LPw}~\ref{it:tts}
  then implies that
  \[
    t^*(H[U\setminus W],w)\ge t(H[U\setminus W],w)\ge
    (1-\alpha)\frac{|U\setminus W|-0.11n/t}{t}. 
  \]  
  By Proposition~\ref{prop:LPw}~\ref{it:restr_w},
  $g\restriction_{U\setminus W}$ and
  $h\restriction_{E(G)\cap \binom{U\setminus W}{2}}$ form a feasible
  solution to the linear programme~\eqref{opt:dual} applied to
  $H[U\setminus W]$ instead of $H$ and, moreover,
  \[
    |g\restriction_{U\setminus W}|+|h\restriction_{E(H)\cap \binom{U\setminus W}{2}}|\ge 
    t^*(H[U\setminus W],w)\ge (1-\alpha)\frac{|U\setminus W| - 0.11n/t}{t}.
  \]
  Thus we have
  \begin{align*}
    t^*(H[U],w)&\ge\left(\sum_{v\in W}g(v)+\sum_{T\in \cT_0}\sum_{uv\in
              E(T)}h(uv)w(uv)\right)
              +|g\restriction_{U\setminus W}|+|h\restriction_{E(H)\cap
              \binom{U\setminus W}{2}}|\\ 
            &\ge (1-\alpha)\left(\frac{|W|}{t-1}+ \frac{|U|- |W| -
              0.11n/t}{t}\right)\\ 
            & \ge(1-\alpha)\frac{|U|+|W|/t-0.11n/t}{t}\\
                &=(1-\alpha)\frac{|U|+0.09n/t}{t} 
              =\frac{|U|}{t}+\frac{0.09n}{t^2} - \alpha
              \frac{1.1n}{t}>{|U|\over t},
  \end{align*}
  which contradicts Proposition~\ref{prop:LPw}~\ref{it:ts_bound}. 
\end{proof}

We denote by $\cT_v$ a family of copies of $K_t$ in $G$ with $v\in T$
for every $T\in \cT_v$ and with $V(T)\cap V(T')=\{v\}$ for all
distinct $T$ and~$T'$ in~$\cT_v$.  The next lemma establishes a
sufficient condition for the existence of a fractional $K_t$-factor in
a weighted graph.

\begin{lemma}
  \label{thm:fract_factor_Gw}
  Let $\alpha\ge 0$ be a real number and let $t$, $n$ and~$D\ge3$ be
  integers with~$\alpha < 1/t^2$ and $D\le n/2$.  Suppose $(G, w)$ is
  an $n$-vertex weighted graph such that {\rm(\textit{i})}~for
  every $v\in V(G)$, there exists a family $\cT_v$ of at least
  $D/(t-1)$ $\alpha$-rich copies of $K_t$ and {\rm(\textit{ii})} for
  every $U\subseteq V$ of size $|U|\ge n-D$, one has
  $t^*(G[U],w)\ge(|U|-D/t)/t$.  Then $(G,w)$ contains a
  fractional $K_t$-factor.
\end{lemma}

\begin{proof}
  Let $g\colon V=V(G)\to\RR_{\ge 0}$
  and~$h\colon E=E(G)\to \RR_{\ge 0}$ form an optimal solution
  to~\eqref{opt:dual}.  If $g(v)>0$ for each $v\in V$, we obtain, by
  Proposition~\ref{prop:LPw}~\ref{it:lower_bound}, that
  $t^*(G,w)\ge |V|/t$.  By
  Proposition~\ref{prop:LPw}~\ref{it:ts_bound},
  we deduce that~$t^*(G,w)=|V|/t$ and conclude that~$(G,w)$ has a 
  fractional $K_t$-factor.
 
  Assume now for a contradiction that there is a vertex $x$ with
  $g(x)=0$.  Let $\cT_x$ be a family of $D/(t-1)$ $\alpha$-rich copies
  of $K_t$.  Let
  $W:=\left(\bigcup_{T\in \cT_x}V(T)\right)\setminus\{x\}$ and observe
  that $|W|=D$. Since the copies of $K_t$ from $\cT_x$ pairwise
  intersect only at~$x$ and they are $\alpha$-rich, it follows that
  \begin{multline*}
    \sum_{v\in W}g(v)+\sum_{T\in \cT_x}\sum_{uv\in E(T)}h(uv)w(uv)\ge \sum_{v\in W}g(v)+\sum_{T\in \cT_x}\sum_{uv\in E(T)}h(uv)(1-\alpha)\\
    \ge (1-\alpha)\left(\sum_{v\in W}g(v)+\sum_{T\in \cT_x}\sum_{uv\in E(T)}h(uv)\right)
    = (1-\alpha)\sum_{T\in \cT_x}\left(\sum_{v\in W\cap V(T)}g(v)+\sum_{uv\in E(T)}h(uv)\right)\\
    \ge (1-\alpha)D/(t-1).
  \end{multline*}

  Proposition~\ref{prop:LPw}~\ref{it:restr_w} yields that
  $g_1:=g\restriction_{V\setminus W}$ and
  $h_1:=h\restriction_{E\cap \binom{V\setminus W}{2}}$ is a feasible
  solution to~\eqref{opt:dual} with~$G$ replaced by~$G[V\setminus W]$.
  Therefore we have $|g_1|+|h_1|\ge t^*(G[V\setminus W],w)$.  By
  assumption~(\textit{ii}), we get
  \[
    |g_1|+|h_1|\ge t^*\left(G[V\setminus W],w\right)\ge
    \frac{|V\setminus W|-D/t}{t}=\frac{n-D-D/t}{t}.
  \]
  Together we obtain 
  \begin{align*}
    t^*(G,w)
    &=|g|+|h|\ge \sum_{v\in W}g(v)+\sum_{T\in \cT_x}\sum_{uv\in
      E(T)}h(uv)w(uv)+|g_1|+|h_1|\\ 
    &\ge \frac{(1-\alpha)D}{t-1}+\frac{n-D-D/t}{t}>{n\over t}.
  \end{align*}
  This contradicts $t^*(G,w)\le n/t$ (see
  Proposition~\ref{prop:LPw}~\ref{it:ts_bound}). Thus, this case never
  happens; i.e., $g(v)>0$ for all $v\in V$.  We conclude that~$(G,w)$
  contains a fractional $K_t$-factor.
\end{proof}

The following corollary of Lemmas~\ref{lem:large_case_tw}
and~\ref{thm:fract_factor_Gw} is our main tool for finding fractional
$K_t$-factors.

\begin{corollary}
  \label{cor:ff}
  Let $\alpha\ge 0$ be a real number and let $t$, $n$ and~$D\ge 3$ be
  integers with~$\alpha < 1/(7t^2)$ and $D\le n/2$.  Suppose $(G,w)$
  is an $n$-vertex weighted graph and suppose~$H$ is the spanning
  subgraph of~$G$ formed by the $\alpha$-rich edges of~$(G,w)$.
  Suppose
  \begin{enumerate}[label=\nlabel]
  \item for every $v\in V(G)$ there exists a family $\cT_v$ of at
    least $D/(t-1)$ $\alpha$-rich copies of $K_t$; \label{cor:1}
  \item every $0.11n/t$ vertices span a copy of $K_t$ in $H$; \label{cor:2}
  \item $H$ has property $\cP(t, D, 0.2n, n)$ for some $D\le
    n/2$. \label{cor:3} 
  \end{enumerate}
  Then $(G,w)$ contains a fractional $K_t$-factor. 
\end{corollary}

\subsection{Further useful properties of $(n,d,\lambda)$-graphs}
We will use the following auxiliary results.

\begin{proposition}[Proposition~2.3
  in~\cite{KSS04}]\label{prop:lambda}
  Let $G$ be an $(n,d,\lambda)$-graph with~$d\le n/2$.  Then
  $\lambda\ge \sqrt{d/2}$.
\end{proposition}

\begin{fact}
  \label{fact:d_bound}
  Let $G$ be an $(n,d,\lambda)$-graph with $d\le n/2$.  Suppose
  $\lambda\le d^{t-1}/n^{t-2}$ for some $t\ge 3$.  Then
  $d\ge n^{1-1/(2t-3)}/2$.
\end{fact}
\begin{proof}
  Proposition~\ref{prop:lambda} tells us that $\lambda\ge \sqrt{d/2}$.
  Thus $\lambda\le d^{t-1}/n^{t-2}$ implies that
  $d^{2t-3}\ge n^{2t-4}/2$, whence
  $d\ge n^{1-1/(2t-3)}/2^{1/{(2t-3)}}$ follows.
\end{proof}

Given two graphs $G$ and $G'$ on the same vertex set $V$, let
$G\setminus G'=(V, E(G)\setminus E(G'))$.  The following proposition
gives a rough estimate for the number of copies of $K_{t-1}$ in
induced subgraphs of $(n,d,\lambda)$-graphs, even after removing a
small number of edges incident to each vertex.

\begin{proposition}
  \label{prop:Kt_in_U}
  For any integer $t\ge 3$, there exists $n_0$ such that every
  $(n,d,\lambda)$-graph $G$ with $n\ge n_0$ satisfies the following.
  Suppose $\lambda (4n/d)^{t-2}\le m\le d$.  Let $G'$ be a graph on
  $V(G)$ with maximum degree $(d/(4 n))^{t-2}m$.  Then, for any
  $2\le i\le t-1$ and any set $U$ of at least $(d/(4n))^{t-i-1}m$
  vertices of $G$, the number of copies of $K_i$ in
  $(G\setminus G')[U]$ is at least
  $2^{-i^2}i!^{-1}{|U|}^{i}(d/n)^{\binom{i}{2}}$ and at most
  $2^{i^2}i!^{-1}{|U|}^{i}(d/n)^{\binom{i}{2}}$.
\end{proposition}

\begin{proof}
  Let $m_i = (d/(4n))^{t-i-1}m$ for $2\le i\le t-1$.  Hence
  $\lambda \le m (d/(4n))^{t-2} = m_i (d/(4n))^{i-1}$.  Suppose
  $|U|\ge m_i$.  We prove by induction on~$i$ that the stated
  estimates hold.  Note that, since $i\ge 2$, we have
  $\Delta(G')\le (d/(4 n))^{t-2}m\le d m_i/(4n)\le d |U|/(4n)$.  Let
  first~$i=2$.  Theorem~\ref{thm:EML} implies that
  $\big|2e_G(U) - (d/n)|U|^2\big|\le\lambda |U|$.  Since
  $\lambda \le |U|(d/(4n))$, we have
  $(3/8)(d/n)|U|^2\le e_G(U)\le (5/8)(d/n)|U|^2$.  Hence the number of
  edges in $(G\setminus G')[U]$ is at most $(5/8)(d/n)|U|^2$ and at
  least $(3/8)(d/n)|U|^2 - |U|\cdot d |U|/(4n) = (1/8)(d/n)|U|^2 $,
  which verifies our claim for~$i=2$.
  Now suppose $3\le i\le t-1$ and that the estimates hold for smaller
  values of~$i$.  Note that, in particular, we have $t\ge 4$.  Let
  $X_1$ be the set of vertices $v\in U$ such that
  $\deg(v, U)\ge 2d|U|/n$.  By the definition of $X_1$ and
  Theorem~\ref{thm:EML}, we have
  \[
    \frac{2d}{n} |U| |X_1| \le e(U, X_1)  \le \frac dn  |U| |X_1| +
    \lambda \sqrt{|U| |X_1|}. 
  \]
  Together with $\lambda \le m_i (d/(4n))^{i-1}\le |U| (d/(4n))^{i-1}$
  and $i\ge 3$, this implies that
  \[
    |X_1|\le \left( \frac{\lambda n}d\right)^2\frac1{|U|} \le
    \frac{|U|}{16} \left(\frac d{4n}\right)^{2i-4} \le \frac{|U|}{16}
    \left(\frac d{4n}\right)^{i-1}.
  \]
  Note that $2d|U|/n \ge 2d m_i/n \ge m_{i-1}$.  By the inductive
  hypothesis, the number of copies of $K_{i}$ in $(G\setminus G')[U]$
  is at most
  \[
    \frac1{i}\left( |X_1|\cdot
      \frac{2^{(i-1)^2}}{(i-1)!}{|U|}^{i-1}\left( \frac dn
      \right)^{\binom{i-1}{2}} + |U| \cdot
      \frac{2^{(i-1)^2}}{(i-1)!}\left(\frac{2d|U|}{n}\right)^{i-1}\left(
        \frac dn \right)^{\binom{i-1}{2}} \right) \le
    \frac{2^{i^2}}{i!} |U|^{i}\left( \frac dn \right)^{\binom{i}{2}},
  \]
  where the $1/i$ factor avoids counting~$i$ times each copy of~$K_i$.

  Similarly, let $X_2$ be the set of vertices $v\in U$ such that
  $\deg(v, U)\le d|U|/(2n)$.  By the definition of $X_2$ and
  Theorem~\ref{thm:EML}, we have
  \[
    \frac dn  |U| |X_2| - \lambda \sqrt{|U| |X_2|} \le e(U, X_2)\le
    \frac{d}{2n} |U| |X_2|.  
  \]
  This implies that $|X_2|\le (2\lambda n/d)^2/|U| \le |U|/4$.
  Note that $d|U|/(4n) \ge d m_i/(4n) = m_{i-1}$.
  By the inductive hypothesis, the number of copies of $K_i$ in
  $(G\setminus G')[U]$ is at least 
  \[
    \frac1i \cdot(|U|-|X_2|)
    \frac{1}{2^{(i-1)^2}(i-1)!}\left(\frac{d|U|}{2n} -
      \frac{d|U|}{4n}\right)^{i-1}\left( \frac dn
    \right)^{\binom{i-1}{2}} \ge \frac{1}{2^{i^2}i!} |U|^{i}\left(
      \frac dn \right)^{\binom{i}{2}}
  \]
  (the~$1/i$ factor takes care of the fact that each copy of $K_i$ is
  counted at most $i$ times).
\end{proof}

We remark that a better estimate on the number of copies of~$K_i$ in
$m$-vertex subgraphs of $(n,d,\lambda)$-graphs~$G$, of the form
$(1+o(1))\binom m i(d/n)^{\binom{i}{2}}$, was obtained by Alon
(see~\cite[Theorem~4.10]{KS06}).  Furthermore, a Tur\'an-type result
was proven by Sudakov, Szab\'o and Vu~\cite{SSV} for the containment
of a copy of~$K_i$ in dense enough subgraphs of
$(n,d,\lambda)$-graphs.  However, both results require a stronger
condition on~$\lambda$ than the one we shall have available in our
applications of Proposition~\ref{prop:Kt_in_U} (see
Sections~\ref{sec:dense} and~\ref{sec:sparse}).

\section{Proof outline}\label{sec:outline}
In the following we provide a proof overview in the case of triangles,
since the general case is similar.  Our arguments combine tools from
linear programming with probabilistic techniques.  In fact, they can
be seen as a synthesis of some methods in Alon, Frankl, Huang, R\"odl,
Ruci\'nski and Sudakov~\cite{AFHRRS12} and in Krivelevich, Sudakov and
Szab\'o~\cite{KSS04}.

Let an $(n,d,\lambda)$-graph~$G$ with $\lambda\le c d^2/n$ be given.
From the expander mixing lemma, Theorem~\ref{thm:EML}, it follows that
every vertex of~$G$ lies in
$\frac{1}{2}\left(d^3/n\pm\lambda d\right)=(d^3/2n)\left(1\pm
  c\right)$ triangles.  The na\"ive greedy approach outlined in the
introduction (see the discussion soon after Theorem~\ref{thm:main})
does not guarantee a collection of $(1-o(1))n/3$ vertex-disjoint
triangles.  Another attempt would be to apply some theorem that would
tell us that the $3$-uniform hypergraph~$\cK_3(G)$ of the triangles
in~$G$ contains an almost perfect matching.  A theorem of Pippenger
(see~\cite{fueredi88:_match}) would do if we knew that~$\cK_3(G)$ is
pseudorandom enough (roughly speaking, one needs that~$\cK_3(G)$
should be approximately $\ell$-regular for some~$\ell\to\infty$ and
that pairs of vertices of~$\cK_3(G)$ should be contained in~$o(\ell)$
triples of~$\cK_3(G)$ (i.e., the `codegrees' should be small)).
However, for~$c$ an absolute constant, this property of~$\cK_3(G)$
cannot be deduced.

We circumvent the fact that~$\cK_3(G)$ is not necessarily pseudorandom
enough by finding a subhypergraph~$H$ of~$\cK_3(G)$ in which the
`deviation' of the number of triangles at any vertex is `smoothed out'
(thus~$H$ will be almost $\ell$-regular).
This can be done if~$G$ has $\ell=n^{\Theta(1)}$ fractional
$K_3$-factors $f_1,\dots, f_\ell$ such that
$\sum_{i=1}^\ell f_i(T)\le 1$ for each $T\in \cK_3(G)$ and, for any
edge $e\in E(G)$, the sum of the weights on the triangles containing
$e$ across~$f_1,\dots,f_\ell$ is at most $\ell^{1-\gamma}$ for some
$\gamma\in (0,1)$.  This latter condition helps us force small
codegrees.

Indeed, with these fractional $K_3$-factors, we can
select~$H\subset\cK_3(G)$ at random, by including each $T\in \cK_3(G)$
in~$H$ independently with probability $\sum_{i=1}^\ell f_i(T)$.  Then
Chernoff's inequality guarantees that~$H$ satisfies, with high
probability, the assumptions of Theorem~\ref{thm:Pippenger}, which is
a packing result from~\cite{KoRo} strengthening the aforementioned 
result of Pippenger.  Such a `randomization' strategy has previously
been successfully employed in~\cite{AFHRRS12} in the context of
perfect matchings in hypergraphs.

Thus, it suffices to find such fractional $K_3$-factors
$f_1,\dots, f_\ell$.
In fact, we find such~$f_i$ with the property
that, for any $e\in E(G)$, we have
$\sum_{e\in E(T)}\sum_{i=1}^\ell f_i(T)\le 1$ (hence
$\sum_{i=1}^\ell f_i(T)\le 1$ for each $T\in \cK_3(G)$ is
automatically true).

Theorem~\ref{thm:main} is vacuously true for~$d=o(n^{2/3})$ when~$t=3$
(owing to Fact~\ref{fact:d_bound}).  We thus
suppose~$d=\Omega(n^{2/3})$.  We consider two cases.  We pick any
$\beta\in (0, 1/3)$ independent of~$n$.  Our first approach
(Theorem~\ref{cor:rand_I}) works as long as $d$ is not too small, say,
$d\ge n^{(2/3)+\beta}$.  In contrast, the second approach
(Theorem~\ref{cor:rand_II}) works as long as $d$ is not too large,
say, $d\le n^{1-\beta}$.

In the first approach, we consider edge-weighted graphs and we
repeatedly `remove' fractional $K_3$-factors from~$G$ (removing from
edges~$e$ the weights of the triangles~$T$ with $e\subseteq V(T)$).  
This is done by Theorem~\ref{cor:rand_I} below, in which we show that
we can repeatedly apply Corollary~\ref{cor:ff} in the remaining
weighted graph $n^\beta$ times.   

When $d$ is close to $n^{2/3}$, our approach above fails because we
cannot execute it sufficiently many times.  To circumvent this, we
randomly split~$E(G)$ into $\ell=n^{\Omega(1)}$ sets
$E_1,\dots,E_\ell$, with each subgraph~$G_i:=(V,E_i)$ distributed as a
random subgraph~$G_p$ of~$G$, where each edge is included in~$G_p$
with probability~$p=1/\ell$, independently of all the other edges.
Then we show that with high probability each~$G_i$ satisfies the
assumptions of Corollary~\ref{cor:ff}, and thus contains a fractional
$K_3$-factor~$f_i$.  This second approach works only
for~$d\le n^{1-o(1)}$, which makes both approaches necessary.

\section{Fractional $K_t$-factors: the dense case}
\label{sec:dense}
In this section we prove the following theorem.

\begin{theorem}
  \label{cor:rand_I}
  For any integer $t\ge 3$ and $\beta\in (0, 1/(2t-3)]$, there exists
  $n_0$ such that every $(n,d,\lambda)$-graph $G$ with $n\ge n_0$,
  $d\ge n^{1-1/(2t-3)+\beta}$ and
  $\lambda\le (1/(20t\cdot 4^{t-2}))d^{t-1}/n^{t-2}$
  contains~$\ell=n^{\beta}$ fractional $K_t$-factors
  $f_1,\dots, f_\ell$ such that
  \begin{equation}
    \label{eq:1}
    \sum_{i=1}^\ell \sum_{T\in\cK_t(G)\colon V(T)\supset \{u,v\}}
    f_i(T)\le 1
    \text{ for every edge }uv\in E(G).
  \end{equation}
\end{theorem}

\subsection{Proof idea}

Our main idea for proving Theorem~\ref{cor:rand_I} is to view the
$(n,d,\lambda)$-graph $G=(V,E)$ as a graph equipped with the weight
function $w\colon E\to[0,1]$.  Once we manage to find a fractional
$K_t$-factor $f$ in $(G,w)$ (by Corollary~\ref{cor:ff}), we update the
weight function $w$ as follows:
$w(uv):=w(uv)-\sum_{T\in \cK_t(G)\colon V(T)\supset\{u,v\}} f(T)$,
which remains non-negative by condition~\eqref{eq:ftf_weight} from
Section~\ref{sec:LP}.  Moreover, the weighted degree of every vertex
decreases by \emph{exactly} $t-1$, since
$\sum_{u\in N(v)}\sum_{T\in \cK_t(G)\colon V(T)\supset\{u,v\}}
f(T)=(t-1)\sum_{T\in\cK_t(G)\colon V(T)\ni v} f(T)=t-1$.  Therefore,
if in our graph $(G,w)$ all weighted degrees were the same, then,
after updating $w$, the weighted degrees stay the same.  To prove
Theorem~\ref{cor:rand_I}, it suffices to show that we can iterate this
procedure $\ell=n^{\beta}$ times.


\subsection{Cliques in weighted subgraphs of $(n,d,\lambda)$-graphs}  

In the next two propositions, we use Proposition~\ref{prop:Kt_in_U} to
derive the assumptions of Corollary~\ref{cor:ff}.  Recall that, given
a graph~$G$ and~$v\in V(G)$, we denote by~$\cT_v$ a family of copies
of $K_t$ in $G$ with $v\in T$ for every $T\in \cT_v$ and with
$V(T)\cap V(T')=\{v\}$ for all distinct $T$ and~$T'$ in~$\cT_v$.

\begin{proposition}
  \label{prop:w_triangles}
  For any integer $t\ge 3$, there exists $n_0$ such that every
  $(n,d,\lambda)$-graph $G$ with $n\ge n_0$ and
  $\lambda\le (1/(20t\cdot 4^{t-2}))d^{t-1}/n^{t-2}$ satisfies the
  following.  Let $G'$ be a graph on $V(G)$ with maximum degree at
  most $(d/(4 n))^{t-2} d/(20t)$.  Then for every $v\in V(G)$ there
  exists a family $\cT_v$ of at least $d/(2t-2)$ copies of $K_t$ in
  $G\setminus G'$.
\end{proposition}

\begin{proof}
  Fix a vertex $v\in V(G)$ and let $U=N_{G\setminus G'}(v)$.  Then
  $|U|\ge (1-d/(80t n))d \ge 0.9d$.  Let $\cT_v$ be a collection of
  copies of $K_t$ in $G\setminus G'$ such that $|\cT_v|< d/(2t-2)$.
  Then
  $|U\setminus \bigcup_{T\in \cT_v} V(T)|\ge 0.9d - d/2 \ge d/(20t)$.
  By Proposition~\ref{prop:Kt_in_U},
  $G[U\setminus \bigcup_{T\in \cT_v} V(T)]$ contains a copy of
  $K_{t-1}$ in $G\setminus G'$, which, together with~$v$, gives a copy
  of $K_t$ in $G\setminus G'$.  The proposition follows.
\end{proof}

The following proposition establishes Property $\cP(t, D, D', n)$ from
Section~\ref{sec:P_ndl} in subgraphs of $(n,d,\lambda)$-graphs for
certain values of the parameters.

\begin{proposition}
  \label{prop:many_triangles_w}
  For any integer $t\ge 3$, there exists $n_0$ such that every
  $(n,d,\lambda)$-graph $G$ with $n\ge n_0$ and
  $\lambda\le (1/(20t\cdot 4^{t-2}))d^{t-1}/n^{t-2}$ satisfies the
  following.  Let $G'$ be a graph on $V(G)$ with maximum degree at
  most $(d/(4 n))^{t-2} d/(20t)$.  Then every $0.11n/t$ vertices of
  $G\setminus G'$ span a copy of $K_t$ in~$G\setminus G'$.   Moreover,
  $G\setminus G'$ has property $\cP(t, d/2, 0.2n, n)$.
\end{proposition}

\begin{proof} 
  We first prove that $G\setminus G'$ has property
  $\cP(t, d/2, 0.2n, n)$.  Set $c:=1/(20t\cdot 4^{t-2})$ so that
  $\lambda\le c d^{t-1}/n^{t-2}$.  Let $\cT_0$ be a family of copies
  of $K_t$ in $G\setminus G'$ of maximum cardinality satisfying
  properties~\eqref{pr:one}--\eqref{pr:three} in the definition of
  $\cP(t, d/2, 0.2n, n)$.  If $|\cT_0|<0.2n/(t-1)$, then let
  $W:=\left(\bigcup_{T\in \cT_0}V(T)\right)\setminus U_0$. It follows
  that $|W|=(t-1)|\cT_0|<0.2n$.  Because of the maximality of $\cT_0$,
  there are no copies of $K_t$ in $G\setminus G'$ with one vertex from
  $U_0$ and the other $t-1$ from $U\setminus(U_0\cup W)$.  Note that
  $|U\setminus(U_0\cup W)|\ge n-d/2-d/(2t)-0.2n\ge n/(10t)$.  Take any
  subset $U'\subset U\setminus(U_0\cup W)$ of cardinality $n/(10t)$
  and note that, by Theorem~\ref{thm:EML}, we have
  \[
    e_G(U_0,U')\ge \frac{d}n \cdot \frac{n}{10t} \cdot
    \frac{d}{2t}-\lambda\sqrt{\frac{n}{10t} \cdot \frac{d}{2t}}\ge
    \frac{d^2}{20t^2} - c \frac{d^{t-1}}{n^{t-3}}\sqrt{\frac{d}{20t^2
        n}}\ge \frac{d^2}{30t^2}.
  \]
  Thus, there is a vertex $v\in U_0$ of degree at least $d/(15t)$ into
  $U'$ in $G$ and hence~$v$ is connected to at least
  $d/(15t) - \Delta(G') \ge d/(15t) - d/(80t)\ge d/(20t)$ vertices in
  $U'$ via edges in $G\setminus G'$.  Let
  $R:=N_{G\setminus G'}(v)\cap U'$.  We have~$|R|\ge d/(20t)$.  By
  Proposition~\ref{prop:Kt_in_U}, $(G\setminus G')[R]$ contains a copy
  of $K_{t-1}$, which, together with $v$, gives a copy of $K_t$ in
  $G\setminus G'$.  This contradicts the maximality of~$\cT_0$.  This
  shows that $|\cT_0|\geq0.2n/(t-1)$, as required.

  It remains to prove that every set $U$ of $0.11n/(t-1)$ vertices of
  $G\setminus G'$ spans a copy of~$K_t$ in~$G\setminus G'$.  This can
  be done in a similar way by showing that Theorem~\ref{thm:EML}
  implies that the average degree in $(G\setminus G')[U]$ is at least
  $7d/(80t)$ and then finding a copy of $K_{t-1}$ in the neighborhood
  of a vertex of maximum degree in~$(G\setminus G')[U]$.  We omit the
  details.
\end{proof}

We now are ready to prove Theorem~\ref{cor:rand_I}.

\subsection{Proof of Theorem~\ref{cor:rand_I}}
We set $\alpha:=(d/(4 n))^{t-2}/(20t)$ and choose $n_0$ sufficiently
large.  We start with the graph $G$ and at the beginning we set all
edge weights to one, i.e., $w(e):=1$ for all $e\in E(G)$. We shall
iteratively apply Corollary~\ref{cor:ff} to find $\ell=n^\beta$
fractional $K_t$-factors $f_1,\dots, f_\ell$. In doing so we will
iteratively update the weights of the edges in $G$.  By
Fact~\ref{fact:a_rich} and Propositions~\ref{prop:w_triangles}
and~\ref{prop:many_triangles_w}, the assumptions of
Corollary~\ref{cor:ff} will be satisfied after each iteration, and we
shall be able to find a fractional $K_t$-factor in the weighted graph
at hand.  Recall that the `weighted degree' $\deg_w(v)$ of a vertex
$v\in V=V(G)$ is defined to be $\sum_{u\in N(v)} w(uv)$.  We observe
that this degree is exactly~$d$ at the beginning (when $w\equiv1$).
Then we update the edge weights for all $uv\in E(G)$ in the $i$th
iteration as follows:
\begin{equation}\label{eq:update_rule}
  w(uv):=w(uv)-\sum_{T\in \cK_t(G)\colon V(T)\supset\{u,v\}} f_i(T).
\end{equation}
Observe that the weighted degree of \emph{every} vertex decreased by
$t-1$, since
\[
  \sum_{u\in N(v)}\sum_{{T\in \cK_t(G): \,V(T)\supset\{u,v\}}}
  f_i(T)=(t-1)\sum_{T\in \cK_t(G):\,  V(T)\ni v} f_i(T)=t-1, 
\] 
because~$f_i$ is a fractional $K_t$-factor.  Now suppose
$\sum_{u\in N(v)} w(uv) \ge (1-\alpha^2)d$ for any $v\in V$ throughout
the process.  Let $G'$ be the graph on $V$ consisting of edges of $G$
that are not $\alpha$-rich.  By Fact~\ref{fact:a_rich},
$\Delta(G')\le \alpha d$.  Thus, by
Propositions~\ref{prop:w_triangles} and~\ref{prop:many_triangles_w},
we can apply Corollary~\ref{cor:ff} with $D=d/2$ to $(G,w)$
iteratively, updating the weights
\[
  \frac{d\alpha^2}{t-1} = \frac{d}{400t^2(t-1)} \left( \frac d{4
      n}\right )^{2t-4} = \frac  {d^{2t-3}}{400t^2(t-1)(4n)^{2t-4}}
  \ge \frac{n^{(2t-3)\beta}}{400t^2(t-1)4^{2t-4}} \ge n^{\beta}=\ell 
\]
times.  Because of our update rule~\eqref{eq:update_rule},
condition~\eqref{eq:1} does hold for the~$f_i$ ($1\leq i\leq\ell$)
that we have obtained.
\qed

\section{Fractional $K_t$-factors: the sparse case}
\label{sec:sparse}

In this section we prove the following theorem.

\begin{theorem}\label{cor:rand_II}
  For any integer $t\ge 3$ and $\delta\in (0, 1/(2t-3))$, there exists
  $n_0$ such that every $(n,d,\lambda)$-graph $G$ with $n\ge n_0$,
  $d\le n^{1-\delta}$ and
  $\lambda\le (1/(50t\cdot 4^{t-2}))d^{t-1}/n^{t-2}$ contains
  $\ell=n^{\delta/(4t^2)}$ fractional $K_t$-factors $f_1,\dots,f_\ell$
  such that
  \begin{equation}
    \label{eq:dagger}
    \sum_{i=1}^\ell \sum_{T\in\cK_t(G)\colon V(T)\supset \{u,v\}}
    f_i(T)\le 1\text{ for every edge }uv\in E(G).
  \end{equation}
\end{theorem}

\subsection{Proof overview}
Our proof strategy this time will be to \textit{partition}~$G$
randomly into $\ell=n^{\delta/(4t^2)}$ edge-disjoint subgraphs~$G_i$,
and then to show that, with high probability, each such random
subgraph~$G_i$ satisfies the assumptions of Corollary~\ref{cor:ff} and
thus contains a fractional $K_t$-factor~$f_i$.  Because the host
graphs~$G_i$ of such~$f_i$ are edge-disjoint,
condition~\eqref{eq:dagger} will hold because, for any $T\in \cK_t(G)$,
there is at most one $i\in [\ell]$ such that $f_i(T)>0$.
Note that, in this section, when we use results from
Sections~\ref{sec:LP} and~\ref{sec:P_ndl}, we always use them on
standard graphs, i.e., with~$w\equiv 1$.  In particular, here, every
edge will be $\alpha$-rich for any~$\alpha\geq0$.  For the sake of
definiteness, we always take~$\alpha=0$ in this section.
 
\subsection{Probabilistic lemmas}
Recall that, for a graph $G=(V,E)$, the random subgraph~$G_p$ of $G$
is a spanning subgraph of~$G$ in which each edge from $E$ is included
with probability~$p$, independently of all other edges.  In this
subsection, we show that, for suitable $(n,d,\lambda)$-graphs~$G$, with
high probability~$G_p$ satisfies the assumptions of
Corollary~\ref{cor:ff} and thus contains a fractional $K_t$-factor
(Theorem~\ref{thm:fract_c_factor_Gp} below).  We first show that, with
high probability, $G_p$~satisfies~\ref{cor:1} in
Corollary~\ref{cor:ff}, with $D=p^{\binom{t}{2}}d/4$.

\begin{proposition}
  \label{prop:clique_stars}
  For any $t\ge 3$ and $\delta\in(0,1/(2t-3))$ there exists $n_0$ such
  that the following holds.  Suppose $G$ is an $(n,d,\lambda)$-graph
  with $n\ge n_0$, $d\le n^{1-\delta}$ and
  $\lambda\le(1/(20t\cdot4^{t-2}))d^{t-1}/n^{t-2}$.  Let $p=d^{-\eta}$
  for some $\eta\in(0,1/t^2)$.  Then, with probability at least
  $1-n\exp(-\sqrt d)$, for any vertex $v\in V(G_p)$ there exists a
  family $\cT_v$ in~$G_p$ with $|\cT_v|\geq p^{\binom{t}{2}}d/(4t-4)$.
\end{proposition}
\begin{proof}
  For any $v\in V(G_p)=V(G)$, by Proposition~\ref{prop:w_triangles},
  there is a family $\cT'_v$ in~$G$ with $|\cT'_v|= d/(2t-2)$.
  Let~$X$ be the number of cliques~$K_t$ from~$\cT'_v$ in~$G_p$.
  Since the cliques in~$\cT_v$ are edge-disjoint,
  $X\sim\Bin(d/(2t-2),p^{\binom{t}{2}})$.  By Chernoff's inequality
  (Theorem~\ref{thm:chernoff}),
  $\PP\left[X<\EE X/2\right]<e^{-\EE
    X/12}=\exp\big(-p^{\binom{t}{2}}d/(24t)\big)\le
  \exp(-d^{1-\binom{t}{2}\eta}/(24t))\le \exp(-\sqrt d)$.  The union
  bound over all vertices yields that, with probability at least
  $1-n\exp(-\sqrt d)$, every vertex~$v$ lies in some family~$\cT_v$ in
  $G_p$ of at least $p^{\binom{t}{2}}d/(4t-4)$ copies of~$K_t$.
\end{proof}

We now state and prove a technical lemma that will be required to show
that~$G_p$ is very likely to have property~$\cP(t,D,D',n)$ for certain
values of~$D$ and~$D'$.  The proof of this lemma is based on Janson's
inequality (Theorem~\ref{thm:Janson}).  However, to have `weak enough
dependence' when applying inequality~\eqref{eq:Janson}, we shall have
to employ an additional trick.

\begin{lemma}
  \label{prop:clique_Janson}
  For any $t\ge 3$ and $\delta\in(0,1/(2t-3))$ there exists $n_0$ such
  that the following holds.  Suppose $G$ is an $(n,d,\lambda)$-graph
  with $n\ge n_0$, $d\le n^{1-\delta}$ and
  $\lambda\le (1/(50t\cdot 4^{t-2}))d^{t-1}/n^{t-2}$.  Let
  $p=d^{-\eta}$ for some $\eta\in (0,\delta/(2t^2)]$.  Then, with
  probability $1-2^{-n}$, for every pair of disjoint subsets~$U_0$
  and~$U'$ of~$V(G)$ with $|U_0|=D= p^{\binom{t}{2}}d/(4t)$ and
  $|U'|= {n}/{(10t)}$, there is a clique~$K$ on~$t$ vertices in~$G_p$
  with $|V(K)\cap U_0|=1$ and $|V(K)\cap U'|=t-1$.
\end{lemma}
\begin{proof}
  Let~$c:=1/(50t\cdot 4^{t-2})$.  We shall use
  Proposition~\ref{prop:Kt_in_U} to show that there are sufficiently
  many cliques~$K_t$ in~$G$ with one vertex from~$U_0$ and $t-1$
  from~$U'$, so that after keeping each edge at random, the
  probability that none of the cliques survives is very small.  Taking
  the union bound over all possible choices of~$U_0$ and~$U'$ will
  then finish the proof.  To estimate the survival probability of some
  clique, we shall apply Janson's inequality,
  Theorem~\ref{thm:Janson}.  However, some special care needs to be
  taken, as otherwise the parameter~$\Delta$ in~\eqref{eq:Delta_def}
  may turn out to be too large.  Therefore, we shall restrict out
  attention to certain `nice' cliques between~$U_0$ and~$U'$.

  For a given set $U_0$ we say that a vertex $u$ from $U'$ is
  \emph{bad} \emph{with respect to $U_0$} if
  $\deg_{U_0}(u)\ge 2 Dd/n$, and otherwise we say it is \emph{good}.
  Let $B$ be the set of bad vertices from $U'$ with respect to
  $U_0$. Then $e(B,U_0)\ge 2|B|Dd/n$, whereas Theorem~\ref{thm:EML}
  asserts that
  \[
    e(B,U_0)\le |B|Dd/n+\lambda \sqrt{|B|D}\le |B|Dd/n+cd^{t-1} \sqrt{|B|D}/n^{t-2}.
  \]
  From this we infer that $|B|Dd/n\le cd^{t-1} \sqrt{|B|D}/n^{t-2}$,
  and therefore
  \begin{equation*}
    |B|\le c^2d^{2t-4}/(Dn^{2t-6})\le (4t)c^2d^{2t-5+\binom{t}{2}\eta}/n^{2t-6}<n/(30t),
  \end{equation*}
  by the choice of $\eta<\delta/\binom{t}{2}$ and $c$. Thus, the
  number of good vertices in~$U'$ is
  $|U'\setminus B|\ge n/(15t)$.

  Next we estimate the number of edges $e_G(U_0,U'\setminus B)$
  between $U_0$ and $U'\setminus B$ in $G$ (we omit the subscript~$G$
  whenever it is clear from the context).  Note that
  $Dn/d^2\ge n^\delta p^{\binom t2}/(4t) \ge n^\delta n^{-\eta t^2}
  /(4t) \ge 1$, because $\eta \le \delta/(2t^2)$.  By
  $|U'\setminus B|\ge n/(15t)$ and Theorem~\ref{thm:EML}, we have
  \[
    e(U_0,U'\setminus B)\ge \frac d
    n\frac{Dn}{15t}-\lambda\sqrt{\frac{Dn}{15t}}\ge
    \frac{dD}{15t}-\frac{Dd}{60t}\sqrt{\frac{d^2}{15t Dn}}\ge
    \frac{Dd}{20t},
  \]
  where we used that
  $ \lambda\le cd^{t-1}/n^{t-2}\le c d^2/n<d^2/(60t n)$ and
  $d^2/(Dn)\le 1$.  Thus, since $\Delta(G)\le d$, we have at least
  $D/(40t)$ vertices in~$U_0$ of degree at least~$d/(40t)$ into
  $U'\setminus B$ in~$G$.

  Given a vertex $u\in U_0$ with $\deg_{U'\setminus B}(u)\ge d/(40t)$,
  we call a vertex $v\in N(u)\cap(U'\setminus B)$ \emph{expensive with
    respect to $u$} if $\deg_{N(u)\cap(U'\setminus B)}(v)\ge 4 d^2/n$,
  and otherwise we call it \emph{inexpensive}.  Let~$R$ be the set of
  expensive vertices with respect to\ $u$. Then
  $e\left(R,N(u)\cap(U'\setminus B)\right)\ge 4d^2|R|/n$, whereas
  Theorem~\ref{thm:EML} asserts that
  \begin{equation*}
    e\left(R,N(u)\cap(U'\setminus B)\right)
    \le\deg_{U'\setminus B}(u)|R|\frac d n
    +c\frac{d^{t-1}}{n^{t-2}}\sqrt{\deg_{U'\setminus B}(u)|R|}
    \le \frac{d^2 |R|}{n} + \frac{c d^2}{n}\sqrt{d|R|},
  \end{equation*}
  where we used that $\deg_{U'\setminus B}(u)\le d$.
  Thus we have $d^2|R|/n\le cd^2 \sqrt{d |R|}/n$ and therefore $|R|\le c^2 d$.

  We now introduce the notion of `nice cliques' in $G$.  Let~$K$ be a
  copy of~$K_t$ in~$G$ with $K\cap U_0=\{u\}$ and
  $|K\cap (U'\setminus B)|=t-1$.  We call~$K$ \emph{nice} if
  $\deg_{U'\setminus B}(u)\ge d/(40t)$, and any other vertex
  $v\in V(K)\setminus\{u\}$ is inexpensive with respect to~$u$.  Since
  there are at least $d/(50t)$ such inexpensive vertices in
  $N(u)\cap (U'\setminus B)$, we can deduce the following lower bound
  on the number of nice cliques in~$G$.  For $u\in U_0$ with
  $\deg_{U'\setminus B}(u)\ge d/(40t)$, let $R_u$ be the set of
  inexpensive vertices with respect to $u$ and notice that
  $|R_u|\ge d/(40t) - c^2 d\ge d/(50t)$.  Let $\cC$ be the family of
  nice cliques in~$G$.  Then we have
  \begin{equation*}
    |\cC|=\sum|\cK_{t-1}(R_u)|
    \overset{\text{Proposition~\ref{prop:Kt_in_U}}}{\ge}
    \frac{D}{40t}\frac{(d/(50t))^{t-1}}{2^{(t-1)^2}(t-1)!}\left(d/n\right)^{\binom{t-1}{2}}
    \ge\frac{D d^{t-1}}{2^{(t-1)^2}(50t)^{t} (t-1)!}\left(d/n\right)^{\binom{t-1}{2}},
  \end{equation*}
  where the sum is over all $u\in U_0$ such that $\deg_{U'\setminus B}(u)\ge d/(40t)$.

  We aim to employ next Janson's inequality to show that at least one
  of these nice cliques survives in $G_p$ with `sufficiently high'
  probability.  Let $X$ be the number of nice cliques that are
  contained in~$G_p$.  By the above bound on $|\cC|$,
  \begin{equation}\label{eq:E_nice_C}
    \EE X=p^{\binom{t}{2}}|\cC|
    \ge\frac{p^{\binom{t}{2}}D
      d^{t-1}}{2^{(t-1)^2}(50t)^{t}(t-1)!}\left(d/n\right)^{\binom{t-1}{2}}
    \ge\frac{p^{t(t-1)}d^{t}}{2^{t^2}(50t)^{t}
      t!}\left(d/n\right)^{\binom{t-1}{2}}. 
  \end{equation}
  It remains to estimate the parameter $\Delta$ in Janson's
  inequality.
  
  Let a clique $K$ from $\cC$ be given.  First we estimate the number
  of nice cliques $K'$ with $K\cap K'\cap U_0=\emptyset$ and
  $|K\cap K'\cap U'|\ge 2$.  Note that $|K'\cap U_0|=1$ and denote
  this only vertex in $K'\cap U_0$ by~$v$.  We need to choose at least
  two vertices to lie in the intersection $K\cap K'\cap U'$, and these
  have to be connected to~$v$.  Since every vertex from $K\cap U'$ is
  good, we have at most $2 Dd/n$ choices for~$v$.  Moreover, we need
  to specify further $t-3$ vertices from $U'$ to belong to the copy of
  $K'$ and these have to form a $(t-3)$-clique in $N(v)$.  For
  $t\ge 5$, since $|N(v)|=d\ge (d/4n)^{2}d$, by
  Proposition~\ref{prop:Kt_in_U} with $m=d$ and~$i=t-3$, there are at
  most
  \begin{equation}\label{eq:B}
    \binom{t-1}{2}\cdot \frac{2 Dd}n\cdot 2^{(t-3)^2}
    d^{t-3}(d/n)^{\binom{t-3}{2}}\le (d^{t-1}/n)
    (d/n)^{\binom{t-3}{2}} 
  \end{equation}
  potential nice cliques $K'$.  Note that the estimates
  in~\eqref{eq:B} also hold for $t=3$ and~$4$.

  Next we estimate the number of nice cliques $K'$ with
  $K\cap K'\cap U_0\neq\emptyset$ and $|K\cap K'\cap U'|\ge 1$ for our
  fixed clique~$K$ from~$\cC$.  We have $t-1$ choices for a common
  vertex $x$ from $K\cap K'\cap U'$.  Since~$x$ is inexpensive, its
  common neighborhood with the vertex $y$ from $K\cap K'\cap U_0$ is
  at most~$4d^2/n$.  It remains to estimate the number of possible
  extensions of $x$ and $y$ to a clique of size~$t$.  For this we
  would like to count the number of copies of $K_{t-2}$ on a set of
  size $4d^2/n$. For $t\ge 4$, since $4d^2/n\ge d^2/n$, by
  Proposition~\ref{prop:Kt_in_U} with $m=d$ and $i=t-2$, there are at most
  \begin{equation}\label{eq:A}
    (t-1)\cdot 2^{(t-2)^2} (4 d^2/n)^{t-2}(d/n)^{\binom{t-2}{2}}\le
    (d^{t-1}/n) (d/n)^{\binom{t-3}{2}} 
  \end{equation}
  potential nice cliques $K'$.  Note that the estimates
  in~\eqref{eq:A} also hold for~$t=3$.

  By~\eqref{eq:B} and~\eqref{eq:A}, we get
  $\Delta \le 2 p^{\binom{t}{2}} |\cC|(d^{t-1}/n)
  (d/n)^{\binom{t-3}{2}} = 2 \EE X (d^{t-1}/n)
  (d/n)^{\binom{t-3}{2}}$.  Then Janson's inequality yields
  \begin{align*}
    \PP[X=0]&\overset{\hphantom{\eqref{eq:E_nice_C}}}{\le}\exp\left(-\frac{
              \mathbb{E}[X]^2}{2\Delta} \right)
              \le 
              \exp\left(-\frac{\EE[X]}{4(d^{t-1}/n) (d/n)^{\binom{t-3}{2}}}\right)\\
            &\overset{\eqref{eq:E_nice_C}}{\le}
              \exp\left(-\frac{p^{t(t-1)}d^{t}\left(d/n\right)^{\binom{t-1}{2}}}{
              2^{t^2+2} (50t)^{t} t!\cdot (d^{t-1}/n)\cdot
              (d/n)^{\binom{t-3}{2}}}\right)=
              \exp\left(-\frac{p^{t(t-1)}}{2^{t^2+2} (50t)^tt!} n^2(d/n)^{2t-4}\right).
  \end{align*}
  Note that $p=d^{-\eta}\ge n^{-\eta}\ge n^{-\delta/(2t^2)}$.
  Together with $d\ge n^{1-1/(2t-3)}/2$ (Fact~\ref{fact:d_bound}) and
  $\delta < 1/(2t-3)$, we get
  \[
    \PP[X=0]\le \exp\left(-\frac{p^{t(t-1)}}{2^{t^2+2t-2}(50t)^{t} t!}
      n^{2-\frac{2t-4}{2t-3}}\right)\le
    \exp\left(-n^{1+1/(4t-6)}\right).
  \]
  Taking the union bound over all choices for $U_0$ and $U'$ (of which
  there are at most~$4^n$), we obtain the desired claim with
  probability $1-4^n\exp(-n^{1+1/(4t-6)})\ge 1-2^{-n}$ for~$n$
  large enough.
\end{proof}

Now we show that it is very likely that~$G_p$ satisfies~\ref{cor:2}
and~\ref{cor:3} in Corollary~\ref{cor:ff} with $D=p^{\binom{t}{2}}d/4$
and $\alpha =0$.

\begin{proposition}
  \label{prop:many_cliques}
  For any $t\ge 3$ and $\delta\in(0,1/(2t-3))$ there exists $n_0$ such
  that the following holds.  Suppose $G$ is an $(n,d,\lambda)$-graph
  with $n\ge n_0$, $d\le n^{1-\delta}$ and
  $\lambda\le (1/(50t\cdot 4^{t-2}))d^{t-1}/n^{t-2}$.  Let
  $p=d^{-\eta}$ for some $\eta\in (0,\delta/(2t^2)]$.  Then, with
  probability $1-2^{-n}$, every $0.11n/t$ vertices of~$G_p$ span a
  copy of $K_t$ and $G_p$ has property
  $\cP(t, p^{\binom{t}{2}}d/4, 0.2 n, n)$.
\end{proposition}
\begin{proof}
  Let~$t$ and~$\delta$ as in the statement be fixed.
  Let~$n_0=n_0(t,\delta)$ be given by Lemma~\ref{prop:clique_Janson}.
  Then for any $\eta\in(0,\delta/(2t^2)]$ the conclusion in
  Lemma~\ref{prop:clique_Janson} holds with probability at least
  $1-2^{-n}$.  Write $D=p^{\binom{t}{2}}d/4$.

  Let $\cT_0$ be a family of $t$-cliques in~$G_p$ of maximum
  cardinality satisfying properties~\eqref{pr:one}--\eqref{pr:three}
  in the definition of $\cP(t, p^{\binom{t}{2}}d/4, 0.2 n, n)$ (see
  Section~\ref{sec:P_ndl}).  If $|\cT_0|<{0.2n}/{(t-1)}$, then let
  $W:=\left(\bigcup_{K\in \cT_0}V(K)\right)\setminus U_0$. It follows
  that $|W|=(t-1)|\cT_0|<0.2n$.  Because of the maximality of $\cT_0$
  there are no cliques in $G_p$ with one vertex from $U_0$ and the
  other $t-1$ vertices from $U':=U\setminus(U_0\cup W)$.  Observe that
  \[
    |U'|\ge n-D-|U_0|-|W|\ge 0.5n.
  \]
  But then, by the assertion of Lemma~\ref{prop:clique_Janson}, there
  must be yet another copy of $K_t$ in $G_p$ between~$U_0$ and $U'$,
  contradicting the maximality of $\cT_0$.  Hence
  $|\cT_0|\geq{0.2n}/{(t-1)}$. 

  It remains to prove that every set $U$ of $0.11 n/t$ vertices of
  $G_p$ spans a clique $K_t$. But this is immediate since $D<0.01n/t$
  and thus any two disjoint subsets of $U$ of cardinality $D$ and
  $0.1n/t$ span a copy of $K_t$.
\end{proof}


Next we show that $G_p$ is very likely to contain a fractional
$K_t$-factor.

\begin{theorem}
  \label{thm:fract_c_factor_Gp}
  For any $t\ge 3$ and $\delta\in(0,1/(2t-3))$ there exists $n_0$ such
  that the following holds.  Suppose $G$ is an $(n,d,\lambda)$-graph
  with $n\ge n_0$, $d\le n^{1-\delta}$ and
  $\lambda\le (1/(50t\cdot 4^{t-2}))d^{t-1}/n^{t-2}$.  Let
  $p=d^{-\eta}$ for some $\eta\in (0,\delta/(2t^2)]$.  Then~$G_p$
  contains a fractional $K_t$-factor with
  probability at least $1-n\exp(-\sqrt d)-2^{-n}$.
\end{theorem}
\begin{proof}
  Let~$t$ and~$\delta$ as in the statement be fixed.
  Let~$n_0=n_0(t,\delta)$ be given by
  Propositions~\ref{prop:clique_stars} and~\ref{prop:many_cliques} for
  the given parameters~$t$ and~$\delta$.  Then, for any fixed
  $\eta\in(0,\delta/(2t^2)]$, the conclusions in
  Propositions~\ref{prop:clique_stars} and~\ref{prop:many_cliques}
  hold with probability at least $1-n\exp(-\sqrt d)-2^{-n}$.  Let
  $D:=p^{\binom{t}{2}}d/4$.  The conclusion in our theorem follows
  from Corollary~\ref{cor:ff} (with $\alpha=0$).
\end{proof}

We are now ready to prove Theorem~\ref{cor:rand_II}.

\subsection{Proof of Theorem~\ref{cor:rand_II}}
Let $p=d^{-\eta}$ where~$\eta=\delta/(2t^2)$.
Let~$\ell:=d^\eta=p^{-1}$.  By Fact~\ref{fact:d_bound}, we have
$d\ge n^{1-1/(2t-3)}/2\ge n^{1/2}$.  Thus
$\ell=d^{\eta}\ge n^{\delta/(4t^2)}$.  Consider the random
variable~$I$ that takes values uniformly at random
in~$[\ell]=\{1,\dots,\ell\}$.  For each~$e\in E(G)$, let~$I_e\sim I$
be an independent copy of~$I$.  We randomly partition the edge set
of~$G$ into spanning subgraphs $G_1,\dots,G_\ell$ of~$G$, where
each~$e\in E(G)$ is put into~$G_{I_e}$.  Observe that each~$G_i$ is
distributed as~$G_p$.  By Theorem~\ref{thm:fract_c_factor_Gp}, the
probability that a given~$G_i$ should \textit{not} contain a
fractional $K_t$-factor is at most $n\exp(-\sqrt d)+2^{-n}$.  Since
$d\ge n^{1-1/(2t-3)}/2$, it follows that, with probability
$1-\ell(n\exp (-\sqrt d)+2^{-n})=1-o(1)>0$, there is a partition
of~$G$ into deterministic edge-disjoint spanning subgraphs
$G_1,\dots,G_\ell$ such that each~$G_i$ ($i\in[\ell]$) possesses a
fractional $K_t$-factor~$f_i$.  Since $G_1,\dots,G_\ell$ are
edge-disjoint, condition~\eqref{eq:dagger} holds because for any
$T\in \cK_t(G)$, there is at most one $i\in [\ell]$ such that
$f_i(T)>0$.\qed

\section{Proof of Theorem~\ref{thm:main}}
\label{sec:main}

Let~$t\geq3$ be fixed and let~$G$ be an $(n,d,\lambda)$-graph as in
the statement of Theorem~\ref{thm:main}.  Let
$\delta=4t^2/((4t^2+1)(2t-3))$ and $\beta=1/((4t^2+1)(2t-3))$, so that
$n^\beta=n^{\delta/(4t^2)}$.  Let~$\ell=n^\beta$.  By
Theorems~\ref{cor:rand_I} and~\ref{cor:rand_II} (depending on whether
$d\ge n^{1-1/(2t-3)+\beta}$ or $d\le n^{1-\delta}$), our graph~$G$
contains fractional $K_t$-factors $f_1,\dots, f_\ell$ such that, for
every edge $uv\in E(G)$, we have
$\sum_{i=1}^\ell \sum_{T\in\cK_t(G)\colon V(T)\supset \{u,v\}}
f_i(T)\le 1$.  Clearly, $f(T):=\sum_{i}^\ell f_i(T)\le 1$ for any
$T\in\cK_t(G)$.  Let $H:=\cK_t(G)$ be the $t$-uniform hypergraph with
copies of~$K_t$ in~$G$ as hyperedges and~$V(H)=V(G)$.  Thus, for every
$v\in V=V(H)$ and $i\in [\ell]$, we have
$\sum_{T\in E(H)\colon v\in T} f_i(T)=1$ and thus
$\sum_{T\in E(H)\colon v\in T} f(T)=\ell$.  Moreover, for any distinct
vertices $u$ and~$v$, we have
\[
  \sum_{T\in E(H)\colon \{u,v\}\subseteq T} f(T) = \sum_{i\in [\ell]}
  \sum_{T\in E(H)\colon \{u,v\}\subseteq T} f_i(T)\le 1.
\]
Let $H_f$ be the random spanning subhypergraph of $H$ such that each
edge $T$ of $H$ is included in~$H_f$ independently with
probability~$f(T)$.

Next we let~$k:=8\beta^{-1/2}$.  For $v\in V$, let~$X_v$ be the degree
of~$v$ in $H_f$.  We have
$\EE X_v=\sum_{T\in E(H)\colon v\in T}f(T)=\ell$.  Then Chernoff's
inequality (Theorem~\ref{thm:chernoff}), implies that
\[
  \PP[|X_v-\ell|>(k/2)\sqrt{\ell\ln\ell}]
  \le2e^{-k^2(\ln\ell)/12}<2e^{-5\ln n}=2n^{-5}.
\] 
On the other hand, for any distinct vertices~$u$ and~$v\in V$, let
$Y_{uv}$ be the collective degree~$\deg_{H_f}(u,v)$ of~$u$ and~$v$
in~$H_f$.  Then
$\EE Y_{uv}=\sum_{T\in E(H)\colon \{u,v\}\subseteq T} f(T)\le 1$ (in
fact, $Y_{uv}=0$ if $uv\notin E(G)$).  By Chernoff's inequality again,
we have $\PP[Y_{uv}\ge 1+3\ln n]\le e^{-3\ln n}=n^{-3}$.  The union
bound over all $u\in V$ and all pairs $\{u,v\}\in \binom{V}{2}$ yields
that, with probability at least $1-1/n$, we have
\[
  \deg_{H_f}(u)=\ell\pm (k/2)\sqrt{\ell\ln\ell} \quad\text{ and }\quad
  \deg_{H_f}(u,v)\le 1+3\ln n
\]
for all~$u$ and~$v\in V$.

It is easy to check now that the assumptions of
Theorem~\ref{thm:Pippenger} are satisfied with $\delta'=1/t$,
$\gamma=0.9$,
$D=\big(1+(k/2)\sqrt{(\ln\ell)/\ell}\big)\ell\le2n^\beta$ and
$C=1+3\ln n$.  Thus an application of that theorem gives us a matching
in $H_f$ covering all but at most
\[
  O\big(n(C/D)^{(1-\delta')/(t-1)}\big)
  =O\left(n \left(\frac{1+3\ln n}{2 n^\beta} \right)^{1/t} \right)
  \le n^{1-1/(8t^4)}
\]
vertices.  These edges correspond to cliques~$K_t$ in the original
$(n,d,\lambda)$-graph~$G$ and thus they correspond to a collection of
vertex-disjoint $t$-cliques covering all but at most $n^{1-1/(8t^4)}$
vertices of the graph.  This completes the proof of
Theorem~\ref{thm:main}. 
\qed

\section{Concluding remarks}\label{sec:remarks}

In this paper we have studied near-perfect $K_t$-factors in sparse
$(n,d,\lambda)$-graphs.  We have presented two different approaches
for finding many `weight-disjoint' fractional $K_t$-factors: one for
`large'~$d$ (Theorem~\ref{cor:rand_I}) and one for `small'~$d$
(Theorem~\ref{cor:rand_II}).

We believe that the first approach is more powerful, since it can be
extended to the whole range of $d$ for all $t\ge 4$ as follows.  In
the proof of Theorem~\ref{cor:rand_I} we used
Proposition~\ref{prop:Kt_in_U} to embed a copy of $K_{t-1}$ in a set
of $\Omega(d)$ vertices, even after removing a small number of edges
from each vertex, say $o((d/n)^{t-2}d)$ many. We can allow to remove
from each vertex even $o(d^2/n)$ many edges if we just ask for the
existence of a single copy of $K_{t-1}$, which would suffice for our
first approach. To establish this, we could have used a Tur\'an-type
result in $(n,d,\lambda)$-graphs due to Sudakov, Szab\'o and
Vu~\cite[Theorem 3.1]{SSV}. However, a drawback is that the
pseudorandomness condition in~\cite[Theorem 3.1]{SSV} is not
numerically explicit (it was stated as $\lambda \ll
d^{t-1}/n^{t-2}$). Thus we chose to present a self-contained and
numerically explicit proof.  Moreover, the first approach seems not to
provide Theorem~\ref{thm:main} for triangle-factors in the whole
range, while triangle-factors can be considered as the most important
testbed at the moment.

As for the second approach, we observe that it can be pushed to work
in a range of~$d$ from $\Omega(n^{1-1/(2t-3)})$ to~$n/\polylog n$,
although near the upper bound we would only be able to assert that one
can cover all but $n/\polylog n$ vertices with vertex-disjoint copies
of~$K_t$.

\subsection*{Recent developments}
When finalizing this paper, we learned that very recently
Nenadov~\cite{Nenadov18} proved the existence of $K_t$-factors in
$(n,d,\lambda)$-graphs with $\lambda\leq\eps d^{t-1}/(n^{t-2} \log n)$
for a suitably small $\eps=\eps_t>0$, which is at most a $\log n$
factor away from Conjecture~\ref{conj:KSS}.  Independently, Morris and
the authors of this paper~\cite{HKMP18} proved a similar result with
$\lambda\leq\eps_td^{t}/n^{t-1}$, which is a stronger requirement,
except when $d\geq c_tn/\log n$ for a suitable constant~$c_t>0$.

An $n$-vertex $(\lambda,p)$-bijumbled graph is a graph that satisfies
inequality~\eqref{eq:EML} of the expander mixing lemma
(Theorem~\ref{thm:EML}). This notion is slightly more general: any
$(n,d,\lambda)$-graph is $(\lambda,d/n)$-bijumbled but in a bijumbled
graph not every vertex has degree $d$.  The result of~\cite{Nenadov18}
applies to $(\eps_tp^2/(n\log n),p)$-bijumbled graphs with minimum degree
$\Omega(pn)$.  The result of the present paper as well as the one
in~\cite{HKMP18} can be easily shown to hold in appropriately
bijumbled graphs as well.



\end{document}